\documentclass[12pt]{amsart}

\title{Quasi-circles through prescribed points}
\author{John M. Mackay}
\address{Mathematical Institute \\
 University of Oxford \\ Oxford, UK.}
\email{john.mackay@maths.ox.ac.uk}
\date{March 12, 2013}
\thanks{JMM was supported by EPSRC grant ``Geometric and analytic aspects of infinite groups''.}

\usepackage{amsmath,amsthm,amsfonts,graphicx,amssymb}
\usepackage{hyperref}

\numberwithin{equation}{section}
\newtheorem{theorem}[equation]{Theorem}
\newtheorem{proposition}[equation]{Proposition}
\newtheorem{corollary}[equation]{Corollary}
\newtheorem{lemma}[equation]{Lemma}

\newtheorem{definition}[equation]{Definition}

\newtheoremstyle{citing}
  {3pt}
  {3pt}
  {\itshape}
  {}
  {\bfseries}
  {}
  {.5em}
  {\thmnote{#3}}

\theoremstyle{citing}
\newtheorem*{varthm}{}

\DeclareMathOperator{\diam}{diam}

\newcommand{\alp}{\alpha}

\newcommand{\del}{\delta}
\newcommand{\gam}{\gamma}
\newcommand{\eps}{\epsilon}

\newcommand{\lam}{\lambda}

\newcommand{\bdry}{\partial_\infty}

\newcommand{\cC}{\mathcal{C}}

\newcommand{\cN}{\mathcal{N}}

\newcommand{\ra}{\rightarrow}
\newcommand{\R}{\mathbb{R}}
\newcommand{\Sph}{\mathbb{S}}
\newcommand{\N}{\mathbb{N}}
\newcommand{\Z}{\mathbb{Z}}
\newcommand{\HH}{\mathbb{H}}

\def\XXint#1#2#3{{\setbox0=\hbox{$#1{#2#3}{\int}$}
\vcenter{\hbox{$#2#3$}}\kern-.5\wd0}}

\numberwithin{equation}{section}

\begin{document}

\begin{abstract}
	We show that in an $L$-annularly linearly connected, $N$-doubling, complete metric space,
	any $n$ points lie on a $\lambda$-quasi-circle, where $\lambda$ depends only on $L, N$ and $n$.
	This implies, for example, that if $G$ is a hyperbolic group that does not
	split over any virtually cyclic subgroup, then any geodesic line in $G$ lies in
	a quasi-isometrically embedded copy of $\HH^2$.	
\end{abstract}

\maketitle

\vspace{-5mm}
\begin{center}
		\emph{Version accepted by Indiana University Mathematics Journal
		\footnote{\url{http://www.iumj.indiana.edu/IUMJ/Preprints/5211.pdf}}}
\end{center}

\section{Introduction}\label{sec-intro}

Menger's theorem for graphs extends to the following topological result,
known as the ``$n$-Bogensatz'' or $n$-arc connectedness theorem.
\begin{theorem}[\cite{Nob-32-arcs-between-points,Zip-33-indep-arcs,Why-48-bogensatz}]\label{thm-bogensatz}
	If $X$ is a connected, locally connected, locally compact metric space 
	that cannot be disconnected by removing any $n-1$ points, 
	then any two points in $X$ can be joined by $n$ arcs, pairwise disjoint apart from their endpoints.
\end{theorem}
A well known corollary of this result is that any $n$ points in $X$ lie on a simple closed curve 
(see Theorem~\ref{thm-circle-n-points} and remark in \cite{ThoVel-08-Mengers-theorem-extension}).

In this paper, we prove analogues of these theorems for quasi-arcs and quasi-circles
using quantitative topological arguments.
Quasi-circles arise naturally in geometric function theory and in the study of boundaries of hyperbolic groups,
and our results have consequences for the geometry of such groups.

\subsection{Statement of results}

In 1963, Ahlfors \cite{Ahl-63-qcircle} showed that a Jordan curve $\gamma \subset \R^2$
is the image of $\Sph^1 \subset \R^2$ under some quasi\-con\-form\-al homeomorphism of $\R^2$
if and only if $\gamma$ is \emph{linearly connected}:
\begin{definition}
	A complete metric space $(X,d)$ is \emph{$L$-linearly connected}, for some $L\geq 1$, if for every $x,y \in X$,
	there exists a continuum $J$ containing $x$ and $y$, so that $\diam(J) \leq L d(x,y)$.
\end{definition}
(This is also called $L$-bounded turning.  Note that in the above definition, we may assume $J$ is an arc.)

More generally, a \emph{quasi-circle} (respectively \emph{quasi-arc}) is a 
quasi\-sym\-met\-ric image of the standard Euclidean circle (respectively interval).
Tukia and V\"ais\"al\"a showed that a metric Jordan curve (i.e., a metric space homeomorphic to $\Sph^1$) is a 
quasi-circle if and only if it is doubling and linearly connected \cite{TV-80-qs}, and likewise for metric Jordan arcs.
Recall that a metric space $(X,d)$ is \emph{$N$-doubling}, for some $N \in \N$, if any ball of radius $r>0$
in $X$ can be covered by $N$ balls of radius $r/2$.

In the remainder of this paper, we define a $\lambda$-quasi-circle (or $\lambda$-quasi-arc) to be a metric Jordan curve 
(or metric Jordan arc) that is doubling and $\lambda$-linearly connected.

The spaces we study have the property that they have no local cut points, in the
following quantitatively controlled sense.
\begin{definition}\label{def-ann-lin-conn}
	Let $(X,d)$ be a metric space.  The annulus around $x$ between radii $r$ and $R$ is denoted by 
	$A(x,r,R) = \overline{B}(x,R) \setminus B(x,r)$.
	
	A metric space $(X,d)$ is \emph{$L$-annularly linearly connected}, for some $L \geq 1$,
	if it is connected, and given $r>0$, $p \in X$, any two points $x, y \in A(p,r,2r)$
	lie in an arc $J$ so that $x, y \in J \subset A(p,r/L,2L r)$.
\end{definition}
(We may assume (on replacing $L$ by $8L$) that $X$ is also $L$-linearly connected.)

Our main theorems are quantitative versions of Theorems \ref{thm-bogensatz} and \ref{thm-circle-n-points}.
\begin{theorem}\label{thm-qcircle-bogensatz} 
	Let $X$ be a $N$-doubling, $L$-annularly linearly connected, and complete metric space.
	For any $n \in \N$, there exists $\lambda = \lambda(L,N,n)$ so that any distinct $x, y \in X$
	can be joined by $n$ different $\lambda$-quasi-arcs, where the concatenation of any two
	forms a $\lambda$-quasi-circle.
\end{theorem}

\begin{theorem}\label{thm-qcircle-n-points}
	Suppose $(X,d)$ is a non-trivial, $N$-doubling, $L$-annularly linearly connected, complete metric space.
	Then any finite set $T \subset X$ lies on a $\lambda$-quasi-circle $\gamma \subset X$,
	where $\lambda = \lambda(L,N,|T|)$.
	Moreover, if $|T| \geq 2$ we can ensure that $\diam(\gamma) \leq \lambda \diam(T)$.	
\end{theorem}

Note that these results apply to the boundaries of many hyperbolic groups.  
For example, we observe the following.
\begin{corollary}
	Suppose $G$ is a $\delta$-hyperbolic group, which does not virtually split over any finite or two-ended subgroup.
	Then any $n$ geodesics in $G$ lie in the image of an $(L,C)$-quasi-isometry $f:\HH^2 \ra G$,
	where $L$ and $C$ depend only on $G$ and $n$.
\end{corollary}
\begin{proof}
	The boundary $\bdry G$, given some fixed visual metric,
	is doubling and annularly linearly connected \cite[Proof of Corollary 1.2]{Mac-10-confdim}.
	Let $x_1, \ldots, x_{2n} \in \bdry G$ be the endpoints of the geodesics.
	
	We apply Theorem~\ref{thm-qcircle-n-points} to find a quasi-circle through $x_1, \ldots, x_{2n}$,
	and extend this to find a quasi-isometrically embedded hyperbolic plane $f:\HH^2 \ra G$
	\cite[Theorems 7.4, 8.2]{BS-00-gro-hyp-embed}.
	Up to modifying $f$ by a finite distance, we may assume that the geodesics lie in the image of $f$.
\end{proof}

\subsection{Background and remarks}
Any two points in a connected, locally connected, locally compact metric space can be joined by an arc. 
The analogous statement for quasi-arcs was proved by Tukia, and is a key tool in this paper.
\begin{theorem}[{\cite[Theorem 1A]{Tuk-96-qarc},\cite[Corollary 1.2]{Mac-08-quasi-arc}}]\label{thm-tukia-qarc-simple}
	Suppose $(X,d)$ is an $N$-doubling, $L$-linearly connected, complete metric space.  Then there exists
	$\lambda=\lambda(N,L)$ so that any two points in $X$ can be connected by a $\lambda$-quasi-arc.
\end{theorem}

For quasi-circles, as far as the author knows, the only non-trivial existence result known prior to this paper is the
following result of Bonk and Kleiner.
(For an analogous statement for certain relatively hyperbolic groups, see \cite[Theorem 1.3]{Mac-Sis-11-relhypplane}.)
\begin{theorem}[{\cite[Theorem 1]{BK-05-quasiarc-planes}}]\label{thm-bdry-group-qcircle}
	If $G$ is a hyperbolic group and its boundary $\bdry G$ is not totally disconnected,
	then $\bdry G$ contains a quasi-circle.
\end{theorem}

Theorem~\ref{thm-bdry-group-qcircle} is motivated by the problem of finding surface subgroups in hyperbolic groups:
undistorted surface subgroups give quasi-isometric embeddings of $\HH^2$ in the group,
and quasi-isometric embeddings of $\HH^2$ exactly correspond to quasi-circles in the boundary.
Theorem~\ref{thm-bdry-group-qcircle} showed that there is no geometric obstruction to finding such a surface subgroup
once the group is not virtually free, answering a question of Papasoglu.

This result is proved by using Theorem~\ref{thm-tukia-qarc-simple}, a dynamical argument, and
Arzel\`a-Ascoli; the indirectness involved means that 
this method cannot show that every point in $\bdry G$ lies in a quasi-circle.

We now consider Theorems~\ref{thm-qcircle-bogensatz} and \ref{thm-qcircle-n-points}.
In these results, we cannot weaken the annular
linear connectedness condition to the topological condition of N\"obeling.
For example, the set $X = \{(x,y) : 0 \leq x \leq 1, |y| \leq x^2\}$ is doubling, linearly connected, and has no local cut points,
but there is no quasi-circle in $X$ that contains the point $(0,0)$.

One might hope for a stronger result than Theorem~\ref{thm-qcircle-n-points},
where rather than a quantitative no local cut points condition, we merely assume
a quantitative version of ``cannot be disconnected by removing $N$ points.''
For example, perhaps in a doubling, LLC, complete metric space, any two points lie on a quasi-circle.

However, our arguments fail in this case, as we strongly use rescaling and Gromov-Hausdorff limits of sequences of spaces.
The LLC(2) condition need not be preserved under such limits: 
consider a sequence of larger and larger circles converging to a line.

On the other hand, Theorem~\ref{thm-qcircle-n-points} is sharp in the following two senses.
First, the hypotheses of this theorem do not suffice to ensure that $x_1, \ldots x_n$ lie on $\gamma$
in the cyclic order given.  
(Consider the closed unit square, which is doubling and annularly linearly connected,
and label the four corners clock-wise $x_1, x_2, x_3, x_4$.
There is no topologically embedded circle containing these points in cyclic order $x_1, x_3, x_2, x_4$.)

Second, $\lam$ must depend on $n$, otherwise one could take increasingly dense subsets
of the sphere and find uniform quasi-circles through these sets.  In the limit, this gives a contradiction.

The key technical tool that we use in this paper is a new method of joining two quasi-arcs together to make a quasi-arc.
This is described in Section~\ref{sec-joining-qarcs}.
The ``quasi-arc $n$-Bogensatz'' Theorem~\ref{thm-qcircle-bogensatz} is proved in Section~\ref{sec-qarcs-bogensatz}.
Finally, we prove Theorem~\ref{thm-qcircle-n-points} in Section~\ref{sec-qcircle-n-points}.

\subsection{Notation}
We denote balls in a metric space $(X,d)$ by $B(x,r) = \{y\in X : d(x,y) < r\}$.
The open neighbourhood of $A \subset X$ of size $r$ is
$N(A,r) = \{y\in X : d(y,A)<r\}$.
If $B = B(x,r)$, and $t>0$, then $tB = B(x,tr)$.
Similarly, if $V = N(A,r)$, and $t>0$, then $tV = N(A,tr)$.

If $C$ is a constant depending only on $C_1,C_2$, then we write
$C=C(C_1,C_2)$.

\subsection{Acknowledgements}
I thank Daniel Groves for bringing this question to my attention, and Bruce Kleiner and Alessandro Sisto for helpful comments.

\section{Joining together quasi-arcs}\label{sec-joining-qarcs}

Any arc in a doubling, linearly connected space can be straightened into a quasi-arc.
\begin{theorem}[{\cite[Theorem 1B]{Tuk-96-qarc}, \cite[Theorem 1.1]{Mac-08-quasi-arc}}]\label{thm-tukia}
 Suppose $(X,d)$ is a L-linearly connected,  N-doubling,
 complete metric space.
 For every arc $A$ in $X$ and every $\eps >0$, there is an arc $J$
 that $\eps$-follows $A$, has the same endpoints as $A$,
 and is an $\alpha\eps$-local $\lambda$-quasi-arc,
 where $\lambda = \lambda(L,N) \geq 1$ and $\alpha = \alpha(L,N) >0$.
 That is, for any $x, y \in J$ with $d(x,y) \leq \alpha\eps$, we have
 $\diam(J[x,y]) \leq \lambda d(x,y)$.
\end{theorem}
The notation we use for arcs is described below.
\begin{definition}[\cite{Mac-08-quasi-arc}]\label{def-iota-follows}
	For any $x$ and $y$ in an embedded arc $A$, let $A[x,y]$ be the closed, possibly trivial,
	subarc of $A$ that lies between them.  Let $A[x,y) = A[x,y] \setminus \{y\}$, and so on.
	
	An arc $B$ {\em $\iota$-follows} an arc $A$ if there
	exists a (not necessarily continuous) map $p:B \rightarrow A$, sending endpoints to endpoints,
	such that for all $x,\,y \in B$, $B[x,y]$ is in the $\iota$-neighbourhood of
	$A[p(x),p(y)]$; in particular, $p$ displaces points at most $\iota$.
\end{definition}

The goal of this section is to refine Theorem~\ref{thm-tukia} to the following situation.
Suppose an arc $I$ is formed from two quasi-arcs joined by an arc $I' \subset I$.
We show how to modify $I$ only near $I'$ to create a quasi-arc.
\begin{theorem}\label{thm-qarc-joining}
	Let $X$ be an $N$-doubling, $L$-linearly connected, complete metric space.
	Let $A$ be an arc formed of three consecutive subarcs $A_1=A[a_0, a_1]$,  $A_2=A[a_1,a_2]$ and
	$A_3=A[a_2,a_3]$.  Suppose that $A_1$ and $A_3$ are $\eps$-local $\lambda$-quasi-arcs, and
	$d(A_1, A_3) \geq 2 \eps$.
	
	Then there exists an $\alpha\eps$-local $\lambda'$-quasi-arc $J$ that $\alpha\eps$-follows $A$,
	for $\alpha = \alpha(L,N,\lambda)>0$ and $\lam'=\lam'(L,N,\lam)>1$.
	Moreover, $J$ contains the initial and final connected components of $A \setminus N(A_2, 2\epsilon)$.
\end{theorem}

This theorem follows the proof of Theorem \ref{thm-tukia} given in \cite{Mac-08-quasi-arc} verbatim,
once we establish the following modified version of \cite[Proposition 2.1]{Mac-08-quasi-arc}.

\begin{proposition}\label{prop-qarc-join-scale}
	We assume the hypotheses of Theorem~\ref{thm-qarc-joining}.
	There exists constants $s=s(L,N,\lambda) >0$ and $S=S(L,N,\lambda)>0$ with the following property:
	for each $\iota \in (0,\eps)$ there exists an arc $J$ that $\iota$-follows $A$, 
	contains the initial and final connected components of $A \setminus N(A_2, 2\iota)$, and satisfies
	\begin{equation} \label{eq-cqa}
		\forall x,y \in J,\  d(x,y) < s\iota \implies
		\diam(J[x,y]) < S\iota \tag{$\ast$}.
	\end{equation}	
\end{proposition}
\begin{proof}
	We modify the proof of \cite[Proposition 2.1]{Mac-08-quasi-arc}.
	To simplify notation, we replace $L$ by $\max\{L, \lambda\}$.
	
	Let $r = \iota / 20L$, and let $\cN$ be a maximal $r$-separated net in $X$ containing $a_0$ and $a_3$.
	Then there exists $\delta = \delta(L,N,\lambda) \in (0,1)$ and
	a collection of sets $\{V_x\}_{x\in \cN}$ so that each $V_x$ is a union of finitely
	many (closed) arcs in $X$, and for all $x, y \in \cN$:
	\begin{enumerate}
		\item[(1)] $d(x,y) \leq 2r \implies y \in V_x$.
		\item[(2)] $\diam(V_x) \leq 5Lr$.
		\item[(3)] $V_x \cap V_y = \emptyset \implies d(V_x, V_y)  > \delta r$.
		\item[(4)] $B(x,r) \cap (A_1 \cup A_3) \subset V_x$.
	\end{enumerate}
	To show this, we follow the proof of \cite[Lemma 3.1]{Mac-08-quasi-arc},
	with the exception that when we construct $V_x^{(0)}$, we also add closed arcs from $B(x,2r) \cap (A_1 \cup A_3)$
	so that the hypotheses of (4) are satisfied, and arcs joining them to $x$ in $B(x,2rL)$.
	Observe that $\diam(V_x^{(0)}) \leq 4Lr$, so the rest of the proof of \cite[Lemma 3.1]{Mac-08-quasi-arc}
	follows unchanged.
	
	Now cover $A_2$ by connected open arcs which lie in some $B(z,r), z \in \cN$, and take a finite subcover of $A_2$.
	Let $y_1, y_2, \ldots, y_m$ be points in $A_2$ lying in the arcs corresponding to this cover, in the order given by $A_2$,
	and let $z_1, z_2, \ldots, z_m \subset \cN$ be the centres of the associated balls.
	
	The collection of sets $\{V_x\}$ is locally finite, and each $V_x$ is compact,
	so there exists a point $q_0 \in A_1$ that is the first point in $A_1$ to be contained
	in some $V_{w_0}$ which meets $\bigcup_j V_{z_j}$.
	
	Let $K$ be the union of $V_x$ so that $V_x \cap A_3 \neq \emptyset$.
	Define $w_i$ inductively as follows, for $i > 0$.
	If $V_{w_{i-1}} \cap K \neq \emptyset$, set $n = i$, and stop.
	Otherwise, let $k_i = \max\{j: V_{w_{i-1}} \cap V_{z_j} \neq \emptyset\}$, set $w_i = z_{k_i}$, and continue.
	
	Finally, let $q_{n+1}$ be the last point in $A_3$ to be contained in some $V_{w_n}$ meeting $V_{w_{n-1}}$.
	By (1), this process is well defined.
	
	We use this sequence to build our path $J$ in stages.
	Set $J_{-1} = A_1[a_0, q_0]$.
	Let $J_0$ be an arc in $V_{w_0}$ that joins $q_0$ to $q_1 \in V_{w_1}$, where
	$J_0 \cap A[a_0, q_0] = \{q_0\}$ and $J_0 \cap V_{w_1} = \{q_1\}$.
	Now for $i$ from $1$ to $n-1$, let $J_i$ be an arc in $V_{w_i}$ that joins $q_i$ in $V_{w_i}$
	to some $q_{i+1} \in V_{w_{i+1}}$, where $q_{i+1}$ is the first point of $J_i$ to meet $V_{w_{i+1}}$.
	Let $J_n$ be an arc in $V_{w_n}$ that joins $q_n$ to $q_{n+1} \in A[q_{n+1}, a_3]$,
	where $J_n \cap A[q_{n+1}, a_3] = \{q_{n+1}\}$.
	To finish, we set $J_{n+1} = A[q_{n+1}, a_3] = A_3[q_{n+1}, a_3]$.
	
	We claim that the arc $J = J_{-1} \cup J_0 \cup \cdots \cup J_{n+1}$ satisfies our conclusions, for suitable $s$ and $S$.
	
	To show that $J$ $\iota$-follows $A$, define a coarse map $f: J \ra A$ as follows.
	If $x \in J_{-1} \cup J_{n+1}$, let $f(x)=x$.
	If $x \in J[q_0, q_1)$, set $f(x)=q_0$, and if $x \in J[q_n, q_{n+1}]$, set $f(x)=q_{n+1}$,
	For $x \in J[q_i, q_{i+1})$, $i=1, \ldots, n-1$, set $f(x) = y_{k_i} \in A_2 \cap B(z_{k_i},r)$.
	
	It is straightforward to check that $f$ satisfies the definition of $\iota$-following.
	For example, suppose $y \in J_i, y' \in J_{i'}$ and $1 \leq i \leq i' \leq n-1$.
	Then
	\begin{align*}
		J[y,y'] & \subset J[q_i, q_{i'+1}] \subset N(\{w_i, \ldots, w_{i'}\}, 5Lr)
			\\ & \subset N(\{y_{k_i}, \ldots, y_{k_{i'}}\}, 5Lr+r)
			\subset N(A[y_{k_i}, y_{k_{i'}}], 5Lr+r)
			\\ & \subset N(A[f(y), f(y')], \iota).
	\end{align*}
	The other cases follow in similar fashion.
	
	As $f$ is the identity on $J_{-1} \cup J_{n+1}$, and $d(q_0, A_2), d(q_{n+1}, A_2) < \iota$,
	$J$ contains the required components of $A_1$ and $A_3$.
	
	All that remains is to show that $J$ satisfies \eqref{eq-cqa}.
	Suppose that $y \in J_i, y' \in J_{i'}$, with $y < y'$ in $J$, and $d(y, y') < r\delta$.
	There are four cases.
	
	(i) If $i=i'=-1$ or $i=i'=n+1$, we have $\diam(J[y,y']) \leq \lambda d(y,y') \leq Lr\delta$.
	
	(ii) If $0 \leq i,i' \leq n$, then $d(V_{w_i}, V_{w_{i'}}) < r\delta$ so by construction and (3) we have
	$|i-i'| \leq 1$ and thus by (2), $\diam(J[y,y']) \leq 10Lr$.
	
	(iii) If $i=-1, i' \geq 0$ then by (4), $y$ lies in some $V_x$, and $d(V_x, V_{w_{i'}}) \leq d(y,y') < \del r$,
	so by (3) and construction, $i'=0$.
	Thus $d(y', q_0) \leq \diam(V_{w_0}) \leq 5Lr$, and $d(y, q_0) \leq 5Lr + d(y,y') < (5L+\del)r$.
	So
	\[
		\diam(J[y,y']) = \diam(J[y,q_0] \cup J[q_0, y']) \leq \lambda(5L+\del)r + 5Lr \leq 11L^2r.
	\]
	
	(iv) The case $i \leq n$, $i' = n+1$ follows similarly to (iii).
	
	We let $s = \del r / \iota = \del/20L$ and 
	$S = \max\{Lr\delta, 10Lr, 11L^2r\}/\iota = 11L/20$, 
	and have proven the proposition.
\end{proof}

\section{Many quasi-arcs between two points}\label{sec-qarcs-bogensatz} 

Our goal in this section is the following theorem.
\begin{varthm}[Theorem \ref{thm-qcircle-bogensatz}]
	Let $X$ be a $N$-doubling, $L$-annularly linearly connected, and complete metric space.
	For any $n \in \N$, there exists $\lambda = \lambda(L,N,n)$ so that any distinct $x, y \in X$
	can be joined by $n$ different $\lambda$-quasi-arcs, so that the concatenation of any two
	forms a $\lambda$-quasi-circle.
\end{varthm}

The key part of this theorem is the following proposition that
splits a quasi-arc into two relatively close and separated quasi-arcs.
This uses arguments similar to \cite[Section 3]{Mac-10-confdim}.

\begin{proposition}\label{prop-qarc-rel-split}
	Given $\lambda_0 \geq 1$, $\eps > 0$, there exists $\lambda = \lambda(L, N, \lambda_0, \eps) \geq 1$
	and $\eta = \eta(L, N, \lambda_0, \eps) > 0$ with the following property:
	
	For any $\lambda_0$-quasi-arc $A = A[a,b]$ in
	an $N$-doubling, $L$-annularly linearly connected, complete metric space $X$,
	there exist two $\lambda$-quasi-arcs $J = J[a,b]$ and $J' = J'[a,b]$ with the following
	properties: 
	
	For all $z \in (J \cup J') \setminus \{a, b\}$,
		\begin{align}
			d(z, A) & \leq \eps\, d(z, \{a, b\}), \text{ and} \label{eq-sep-nbhd} \\
			\max\{ d(z, J), d(z, J') \} & \geq \eta\, d(z, \{a, b\}). \label{eq-rel-sep}
		\end{align}
\end{proposition}
\begin{proof}
	We may rescale so that $d(a,b)=1$, and assume that $\eps < 1$.
	Let $\del = 1/10 \lambda_0$.
	
	We consider $A$ in the natural order from $a$ to $b$.
	For each $i \in \N$, let $x_{-i}$ be the first point in $A$ at distance $\del^i$ from $a$,
	and let $x_i$ be the last point in $A$ at distance $\del^i$ from $b$.
	
	Let $D_1 = \eps\delta/3\lam_0$, and for $i \in \Z \setminus \{0\}$, let $B_i = B(x_i, D_1 \del^{|i|})$.
	
	For $i < 0$ let $A_i = A[x_{i-1},x_i]$, let $A_0 = A[x_{-1}, x_1]$, and 
	for $i > 0$ let $A_i = A[x_i, x_{i+1}]$.
	Set $D_2 = D_1\delta/10\lam_0 L$, and for $i \in \Z$ let $V_i = N(A_i, D_2 \del^{|i|})$.	
	\begin{lemma}\label{lem-qarc-split-nbhds}
		These neighbourhoods have the following properties:
		
		(1) If $i \neq j$, then $B_i \cap B_j = \emptyset$.
		
		(2) If $i<0$ and $j < i$, then $V_{i+1} \cap B_j = V_{i+1} \cap V_j = B_i \cap V_j = \emptyset$.
		
		(3) If $i<0$, then $(V_i \cap V_{i-1}) \subset B_{i-1}$.		
	\end{lemma}
	\begin{proof}
		(1) This is immediate.
		
		(2) This follows from the following claim.
			If for some $i<0$, we have $z \in A[a,x_{i-1}], z' \in A[x_i, x_{i+1}]$ then
			$d(z,z') \geq D_1 \del^{|i-1|}+ D_1 \del^{|i+1|}$, for otherwise
			\[
				\diam(A[z,z']) \leq \lam_0 (D_1 \del^{|i-1|}+ D_2 \del^{|i+1|}) \leq \frac12 \del^{|i|},
			\]
			but
			\[
				\diam(A[z,z']) \geq \diam(A[x_{i-1},x_i]) \geq \del^{|i|} - \del^{|i-1|} \geq \frac23 \del^{|i|}.
			\]
			
		(3) If not, there exist $z \in A_{i-1}, z' \in A_{i}$ outside $\frac12 B_{i-1}$,
			so that $d(z,z') \leq D_2 \del^{|i-1|} + D_2 \del^{|i|}$.
			Therefore 
			\[
				\diam(A[z,z']) \leq \lambda_0 D_2 \del^{|i-1|} (1+\del^{-1}) 
					\leq \frac1{10} D_1 \del^{|i-1|}
			\]
			but $A[z,z']$ must pass through the centre of $\frac12 B_{i-1}$,
			so $\diam(A[z,z']) \geq \frac12 D_1 \del^{|i-1|}$, a contradiction.
	\end{proof}
	
	We now split $A$ into two disjoint arcs along the subarcs $A_{2i}$, $i\in\Z$.
	See Figure \ref{fig-splitarc}.
	\begin{figure}
		\begin{center}
		\includegraphics[width=\textwidth]{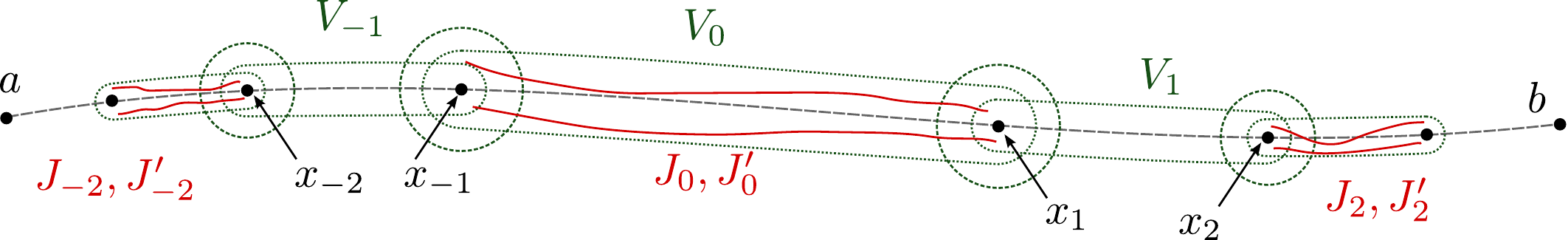}
		\end{center}
 		\caption{Splitting a quasi-arc into a quasi-circle}
		\label{fig-splitarc}
	\end{figure}
	\begin{lemma}\label{lem-qarc-split-every-other}
		For $i \in \Z$ we can find two $\lam_1$-quasi-arcs $J_{2i}, J_{2i}'$ that
		$\frac12 D_2 \del^{|2i|}$-follow $A_{2i}$, and are $\eta_1 \del^{|2i|}$ separated,
		where $\eta_1 = \eta_1(L,N,\lam_0,\eps)>0$ and $\lam_1=\lam_1(L,N,\lam_0,\eps)>1$.
	\end{lemma}
	\begin{proof}
		Using \cite[Lemma 3.3]{Mac-10-confdim} with ``$\eps$'' equal to $\frac14 D_2 \del^{|2i|}$, 
		we split $A_{2i}$ into two $2\eta_1 \del^{|2i|}$ separated arcs that $\frac14 D_2 \del^{|2i|}$-follow 
		$A_{2i}$, for some $\eta_1 = \eta_1(L,N,\eps, \lambda_0) \in (0, \frac14 D_2)$.
		We then apply Theorem~\ref{thm-tukia} to these arcs with ``$\eps$'' equal to 
		$\frac12 \eta_1 \del^{|2i|}$, to get
		two $\frac12 \alpha\eta_1 \del^{|2i|}$-local $\lambda_1'$-quasi-arcs $J_{2i}, J_{2i}'$
		that $\frac12 D_2 \del^{|2i|}$-follow $A[x_1, y_1]$ and are $\eta_1 \del^{|2i|}$-separated,
		for suitable $\alpha=\alpha(L,N)$ and $\lambda_1' = \lambda_1'(L,N)$.
		
		Every $\beta$-local $\mu$-quasi-arc of diameter $D$ is a $\max\{\mu, D/\beta\}$-quasi-arc,
		and these arcs have diameter at most $2\lam_0 \del^{|2i|}$, so $J_{2i}, J_{2i}'$ are
		$\lam_1$-quasi-arcs, for $\lam_1=\lam_1(\alpha, \eta_1, \lam_0, \lam_1')$.
	\end{proof}
	
	Now, following the proof of \cite[Lemma 3.5]{Mac-10-confdim}, for
	each $i \in \Z$, one can join the pair of arcs $J_{2i}, J'_{2i}$ to the arcs $J_{2i-2}, J'_{2i-2}$
	inside $\frac12 B_{2i-1} \cup \frac12 V_{2i-1} \cup \frac12 B_{2i-2}$, with control on the separation
	between the resulting arcs.  
	We prove this in the case $i \leq 0$; $i>0$ is handled similarly.

	The separation properties of Lemma \ref{lem-qarc-split-nbhds} ensure that the following process
	can be applied independently in each location.

	\emph{Topological joining:}
		Join the endpoints of $J_{2i}, J'_{2i}$ to $A$ inside the ball
		$B(x_{2i-1}, \frac12 L D_2 \del^{|2i|}) = ({L D_2}/{2D_1 \del}) B_{2i-1}$.
		Similarly, join the endpoints of $J_{2i-2}, J'_{2i-2}$ to $A$ inside $({L D_2}/{2D_1}) B_{2i-2}$.
		
		We ``unzip'' $A$ along this segment to join the two pairs of arcs $J_{2i}, J'_{2i}$ and 
		$J_{2i-2}, J'_{2i-2}$  by two disjoint arcs $\tilde{J}_{2i-1}, \tilde{J}_{2i-1}'$ in $\frac{1}{4}V_{2i-1}$
		(see \cite[Lemma 3.1, Lemma 3.5]{Mac-10-confdim}).
		This involves discarding the ends of the four given arcs, but all changes
		take place inside $\frac14 B_{2i-2} \cup \frac14 V_{2i-1} \cup \frac14 B_{2i-1}$.
	
	\emph{Quantitative control:}
		As in \cite[Lemma 3.5]{Mac-10-confdim}, compactness arguments ensure that
		$d(\tilde{J}_{2i-1}, \tilde{J}_{2i-1}') \geq 2\eta_2 \del^{|2i-1|}$,
		for some value $\eta_2 = \eta_2(L,N,\eps, \lam_0, \lam_1)>0$.
	
	We now straighten these separated arcs into quasi-arcs.
	
	\emph{Straightening:}
		We assume, after swapping $J_*, J_*'$ if necessary, that $\tilde{J}_{2i-1}$ joins
		$J_{2i}$ and $J_{2i-2}$, and that $\tilde{J}_{2i-1}'$ joins $J'_{2i}$ and $J'_{2i-2}$.
		
		We apply Theorem~\ref{thm-qarc-joining} to $J_{2i-2} \cup \tilde{J}_{2i-1} \cup J_{2i}$
		with ``$\eps$'' equal to $\frac12 \eta_2 \del^{|2i-1|}$, to straighten the arc
		into a $\lam_2$-quasi-arc $J_{2i-2} \cup J_{2i-1} \cup J_{2i}$, making changes
		only in $\frac12 B_{2i-2} \cup \frac12 V_{2i-1} \cup \frac12 B_{2i-1}$.
		Here $\lam_2 = \lam_2(L,N,\eta_2, \lam_1) \geq \lam_1$.
		Again, we may discard ends of $J_{2i}, J'_{2i}, J_{2i-2}, J'_{2i-2}$ in
		$\frac12 B_{2i-2} \cup \frac12 B_{2i}$.
		
	We claim that the arcs $J = \bigcup_{i \in \Z} J_i$ and $J' = \bigcup_{i \in \Z} J'_i$
	satisfy our requirements.
	\begin{lemma}
		$J$ and $J'$ are $\lam$-quasi-arcs, for $\lambda = \lambda(L,N,\lambda_0,\eps)$.
	\end{lemma}
	\begin{proof}
		Suppose $x,y \in J$, where $x \in J_i$ and $y \in J_j$, $i \leq j$.
		It suffices to consider the following three cases.
		
		If $|i-j| \leq 1$, then $\diam(J[x,y]) \leq \lam_2 d(x,y)$.
		
		If $i<0<j$ then $d(a,x), d(b,y) \leq \frac15$, so 
		$\diam(J[x,y]) \leq \diam(J) \leq 2\lam_0 \leq 4\lam_0 d(x,y)$.
		
		If $i+1<j\leq 0$, then $x \in \frac12 B_{i-1} \cup \frac12 V_i \cup \frac12 B_i$,
		and $y \in \frac12 B_{j-1} \cup \frac12 V_j \cup \frac12 B_j$ (where $B_0=B_1$).
		Thus by Lemma~\ref{lem-qarc-split-nbhds}, $d(x,y) \geq \frac12 D_2 \del^{|j|}$, so
		\[
			\diam(J[x,y]) \leq \diam(J[a, x_j]) \leq 2\lam_0 \del^{|j|} \leq \frac{4\lam_0}{D_2} d(x,y).
		\]
		We set $\lam = \max \{ \lam_2, 4 \lam_0 /D_2 \}$, and are done.
	\end{proof}
	
	It remains to check the neighbourhood and separation conditions.
	\begin{lemma}
		$J$ and $J'$ satisfy \eqref{eq-sep-nbhd}.
	\end{lemma}
	\begin{proof}
		It suffices to consider $z \in J_i$, $i \leq 0$.
		If $z \in \frac12 B_{i-1}$, then $d(z,A) < \frac12 D_1 \del^{|i-1|}$, and
		$d(z,a) > \del^{|i-1|} (1-\frac12 D_1) > \frac12 \del^{|i-1|}$.
		Therefore, as $D_1 < \eps$, \eqref{eq-sep-nbhd} holds.
		Similarly, if $z \in \frac12 B_i$, \eqref{eq-sep-nbhd} holds.
		
		It remains to check when $z \in \frac12 V_i$.
		Then there exists some $z' \in A_i$ so that $d(z,z') \leq \frac12 D_2 \del^{|i|}$.
		Thus as $d(a, A_i) \geq \del^{|i-1|}/\lam_0$, 
		\[
			d(z,a) \geq d(z',a) - \frac12 D_2\del^{|i|} \geq \del^{|i|} (\frac{\del}{\lam_0} -\frac{D_2}{2})
				\geq \frac{1}{20\lam_0^2} \del^{|i|}.
		\]
		Since $d(z,A) \leq \frac12 D_2 \del^{|i|}$, and ${20\lam_0^2 D_2}/{2} \leq \eps$, we are done.
	\end{proof}
	
	\begin{lemma}
		$J$ and $J'$ satisfy \eqref{eq-rel-sep}, for some $\eta = \eta(L,N,\eps,\lam_0)$.
	\end{lemma}
	\begin{proof}
		Suppose $z \in J_i \subset J$, for $i \leq 0$.
		Let $z' \in J'$ be the closest point to $z$.
		If $z' \in J'_{i-1} \cup J'_i \cup J'_{i+1}$, then
		$d(z,z') \geq \min \{\eta_1 \del^{|i|}, \eta_2 \del^{|i-1|} \}$.
		Otherwise, by Lemma~\ref{lem-qarc-split-nbhds},
		$d(z,z') \geq \frac12 D_2 \del^{|i|}$.
	\end{proof}
	
	This completes the proof of Proposition~\ref{prop-qarc-rel-split}.
\end{proof}

Observe that the relative separation condition \eqref{eq-rel-sep} proven above 
suffices to show that we have a quasi-circle.
\begin{lemma}\label{lem-qarc-rel-sep-makes-qcircle}
	If $J, J'$ are two $\lambda$-quasi-arcs with the same endpoints $a,b$, 
	and satisfying \eqref{eq-rel-sep} for some $\eta \in (0,1)$,
	then $\gam = J \cup J'$ is a $6\lambda / \eta$-quasi-circle.	
\end{lemma}
\begin{proof}
	Clearly $\gam$ is a topological circle.
	Let $x,y \in \gam$ be two points we wish to check for linear connectivity.
	The only non-trivial case is when (up to relabelling) $x \in J \setminus \{a, b\}$ and
	$y \in J' \setminus \{a, b\}$.
	
	First suppose that $d(\{x,y\}, \{a, b\}) \geq \frac12 d(a,b)$,
	Then by \eqref{eq-rel-sep}, we have $d(x,y) \geq \eta \cdot \frac12 d(a,b)$,
	so
	\begin{equation}\label{eq-simple-diam-bound}
		\diam(\gam[x,y]) \leq \diam(\gam) \leq 2\lambda d(a,b) \leq \frac{4\lambda}{\eta} d(x,y),
	\end{equation}
	where $\gam[x,y]$ denotes an appropriate subarc of $\gam$.
	
	Otherwise, we may suppose that $d(x,a) \leq \frac12 d(a,b)$, and so $d(x,y) \geq \eta \cdot d(x,a)$.
	If $d(x,y) \geq \frac13 d(a,b)$ then as in \eqref{eq-simple-diam-bound} we have $\diam(\gam) \leq 6\lam d(x,y)$.
	So we assume that $d(x,y) \leq \frac13 d(a,b)$,
	giving that $d(y,a) \leq \frac56d(a,b)$, hence $d(y,b) \geq \frac15 d(y,a)$.
	Thus $d(x,y) \geq \eta \cdot \frac15 d(y,a)$.
	Therefore,
	\begin{align*}
		\diam(\gam[x,y]) & \leq \diam(J[a,x]) + \diam(J'[a,y]) 
		\leq \lambda d(a,x) + \lambda d(a,y) \\
		& \leq \frac{\lambda}{\eta} d(x,y)+\frac{5\lambda}{\eta} d(x,y) = \frac{6\lambda}{\eta} d(x,y). \qedhere
	\end{align*}
\end{proof}

We now complete the proof of the ``quasi-arc $n$-Bogensatz.''
\begin{proof}[Proof of Theorem~\ref{thm-qcircle-bogensatz}]
	We may assume that $n = 2^m$.
	
	We claim that by induction on $m$, we can find $\kappa_m = \kappa_m(L, N) \geq 1$ and $\eta_m = \eta_m(L,N) \in (0,1)$
	so that there are $2^m$ different $\kappa_m$-quasi-arcs from $x$ to $y$ that pairwise satisfy \eqref{eq-rel-sep}
	with $\eta = \eta_m$.
	
	The $m=0$ case follows from Theorem~\ref{thm-tukia-qarc-simple}, finding a $\kappa_0$-quasi-arc
	between $x$ and $y$, where $\kappa_0 = \kappa_0(L,N)$.  We set $\eta_0 = 1$.
	
	For the induction step with $m \geq 1$,
	apply Proposition~\ref{prop-qarc-rel-split} to each of the $2^{m-1}$ previous $\kappa_{m-1}$-quasi-arcs,
	with $\eps_m = \frac{1}{4} \eta_{m-1}$.
	This results in $2^m$ different $\kappa_m$-quasi-arcs, pairwise satisfying \eqref{eq-rel-sep} for some
	value of $\eta$ which we denote by $\eta_m$.  (Here $\kappa_m = \kappa_m(L,N,\kappa_{m-1},\eps_m)>1$, and
	$\eta_m = \eta_m(L,N,\kappa_{m-1},\eps_m)>0$.)
	
	Finally, Lemma~\ref{lem-qarc-rel-sep-makes-qcircle} completes the proof.
\end{proof}

\section{Quasi-circles through \emph{n} points}\label{sec-qcircle-n-points}

The following corollary of the $n$-Bogensatz is well known.
To motivate the proof of Theorem~\ref{thm-qcircle-n-points}, we include a short proof.
\begin{theorem}\label{thm-circle-n-points}
	Let $X$ be a connected, locally connected and locally compact metric space.
	If $n \geq 2$ and $X$ is not disconnected by the removal of any $n-1$ points, 
	then any $n$ points in $X$ lie on a simple closed curve.
\end{theorem}
\begin{proof}
	The $n = 2$ case is just a restatement of Theorem~\ref{thm-bogensatz}.
	
	We prove the $n > 2$ case by induction.
	Suppose $x_1, \ldots, x_n$ are given.
	By induction, we can find a simple closed curve $\gamma$ containing $x_1, \ldots, x_{n-1}$;
	we relabel so that they are in the cyclic order $x_1, \ldots x_{n-1}$.
	Let $D$ be a closed disc with centre labelled $x_*$, and choose subsets $\{y_1, \ldots y_n\} \subset \partial D$ and 
	$\{y_1', \ldots, y_n'\} \subset \gamma$.
	Let $Y$ be the topological space formed from $D$ and $X$ by gluing together $y_i$ and $y_i'$ for each $1 \leq i \leq n$.
	
	The space $Y$ is connected, locally connected, locally compact, and cannot be disconnected by the removal of 
	any $n-1$ points.  Thus Theorem~\ref{thm-bogensatz} gives $n$ disjoint arcs 
	$\alpha_1, \ldots, \alpha_n$ from $x_*$ to $x_n$ in $Y$.
	For each $i=1, \ldots, n$, let $\beta_i$ be the closed, connected subarc of $\alpha_i$ which 
	contains $x_n$ and exactly one point $z_i$ of $\gamma$.  
	Each point $z_i$ lies in one of $\gam[x_1,x_2),\gam[x_2,x_3), \ldots \gam[x_{n-1},x_1)$.
	By the pigeonhole principle, two of the points lie in the same interval, 
	and so we use these two $\beta$ arcs to
	find a simple closed curve containing $x_1, \ldots, x_{n}$.
\end{proof}

This proof cannot be used directly in the quasi-arc case: the space $Y$ has local cut points.
Moreover, to apply the straightening techniques of Theorem~\ref{thm-qarc-joining}
we need a quasi-arc of controlled size through each $x_i$.
In adapting this proof, the following corollary of the $n$-Bogensatz, due to Zippin, will be useful.
\begin{theorem}[{\cite[Corollary 9]{Zip-33-indep-arcs}}]\label{thm-zippin}
	Let $X$ be a connected, locally connected, locally compact, separable metric space.
	If $A, B \subset X$ are compact subsets of size at least $n$, and there is no subset $S \subset X$
	of size at most $n-1$ so that $A\setminus S$ and $B \setminus S$ lie in different components of $X \setminus S$,
	then there exists $n$ disjoint arcs joining $A$ and $B$.
\end{theorem}

\begin{proof}[Proof of Theorem~\ref{thm-qcircle-n-points}]
	The $n=1$ and $n=2$ cases follow from Theorem \ref{thm-qcircle-bogensatz}.
	We prove the $n>2$ case by strong induction.
	
	By induction, there exists $\lambda_1 = \lambda_1(L,N,n-1)$ so that any
	set $T$ of at most $n-1$ points in an $N$-doubling, $L$-annularly linearly connected, 
	complete metric space $X$ must lie in a $\lambda_1$-quasi-circle $\gamma$ with $\diam(\gamma) \leq \lambda_1 \diam(T)$.
	
	Suppose $x_1, \ldots, x_{n}$ are given.
	Without loss of generality, we may assume that
	$d(x_1, x_2) \leq d(x_i, x_j)$ for all $i \neq j$, and that
	$d(x_1, x_i) \leq d(x_1, x_{i+1})$ for $i=2, \ldots, n-1$.
	We rescale so that $d(x_1, x_n) = 1$.
	
	Let $s = d(x_1, x_2)$, $S = \diam(\{x_1, \ldots, x_n\}) \in [1, 2]$, and
	set $\del = 1/{200L^2 \lambda_1^3}$.
	
	The proof splits into two cases.
	
	\vspace{2mm}
	{\noindent\emph{Case 1:}} Suppose $s \geq \del^{n-1}$.
	
	By induction, there exists a $\lambda_1$-quasi-circle $\alp_1$ through $x_2, \ldots, x_{n}$
	of diameter at most $2\lambda_1$, and at least $s$.  We relabel $x_2, \ldots, x_{n}$ so that
	they lie in $\alp_1$ in this cyclic order.
	
	Now suppose $d(x_1, \alp_1) \leq s/10L\lam_1$.
	Then one can alter $\alp_1$ using a detour in
	$A(x_1, s/10L^2\lam_1, s/5\lam_1)$ to find a simple closed curve $\alp_2$ which
	does not meet $B(x_1, s/10L^2\lam_1)$.
	Since this only cuts out loops of $\alp_1$ in $B(x_1, s/5)$, $\alp_2$ agrees with $\alp_1$ 
	outside $B(x_1, s/5)$, and is a $\lam_1$-quasi-arc there.
	Therefore we can apply Theorem~\ref{thm-qarc-joining} with $\eps = s/100L^2\lam_1$ to straighten
	$\alp_2$ into a $\lambda_2$-quasi-circle $\beta_1$, which passes through $x_2, \ldots, x_{n}$,
	and does not meet $B(x_1, s/20L^2\lam_1)$, for $\lambda_2 = \lambda_2(L,N,\lam_1,s/S) \geq \lambda_1$.
	
	If $d(x_1, \alp_1) \geq s/10L\lam_1$, then we set $\beta_1 = \alp_1$ and continue.
	
	By the $n=2$ case of the theorem, we find a $\lambda_1$-quasi-circle $\beta_2$ through
	$x_1$ of diameter at least $s/50L^2\lam_1^2$, inside $B(x_1, s/40L^2\lam_1)$.
	
	As $X$ has no local cut points, no two disjoint compacta can be separated by removing any
	finite number of points.  Therefore, Theorem~\ref{thm-zippin} implies that we can join
	$\beta_1$ to $\beta_2$ by $2n$ disjoint arcs inside $B(x_1, 4 L S)$.
	We can control the separation of these arcs.
	
	\begin{lemma}[{Cf.\ \cite[Lemma 3.3]{Mac-10-confdim}}]\label{lem-qarc-join-unif-sep1}
	We can join $\beta_1$ to $\beta_2$ by $2n$ arcs in $B(x_1, 4 \lam_1 L S)$
	that are $\delta_* S$-separated, for  $\delta_* = \delta_*(L,N, \lam_1, \lam_2, s/S) >0$.
	\end{lemma}
	\begin{proof}
	This follows from a compactness argument: if not, there is a sequence of configurations giving
	counterexamples.  To be precise, we can find (on rescaling to $S=1$), a sequence
	\[
		\{ \cC^i = (X^{(i)}, x_1^{(i)}, \beta_1^{(i)}, \beta_2^{(i)}) \}_{i\in\N}
	\]
	so that
	for each $i \in \N$, $X^{(i)}$ is an $L$-annularly linearly connected, $N$-doubling, complete metric space with base point 
	$x_1^{(i)}$, and $\beta_1^{(i)}$ and $\beta_2^{(i)}$ are $\lam_2$-quasi-circles in $B(x_1^{(i)}, 2 \lam_1 S)$,
	with uniformly controlled diameter and separation.
	Moreover, there do not exist $2n$ disjoint arcs connecting $\beta_1^{(i)}$ to $\beta_2^{(i)}$ which 
	are $1/i$ separated.
	
	Such configurations have a subsequence that converges to a limit configuration 
	$(X^\infty, x_1^\infty, \beta_1^\infty, \beta_2^\infty)$ in the Gromov-Hausdorff topology.
	We apply Theorem~\ref{thm-zippin} to the limit space $X^{\infty}$ to find $2n$ disjoint arcs joining
	$\beta_1^\infty$ to $\beta_2^\infty$ inside $B(x_1^\infty, 3 \lam_1 L S)$.
	As these arcs are disjoint, they are separated by some definite distance.
	These arcs will then lift back to $\cC^i$ for sufficiently large $i$ to give a contradiction.
	\end{proof}
	
	Now of these $2n$ arcs, at most $n$ of them can be $\frac{1}{2} \delta_* S$ close to any of the
	$n$ different points $x_1, \ldots, x_n$.
	Therefore, we can find $n$ arcs $\gam_1', \ldots, \gam_n'$ which join $\beta_1$ to $\beta_2$,
	are $\delta_* S$-separated, and have distance at least $\frac12 \delta_* S$ from any $x_i$.
	
	By the pigeonhole principle, two of the arcs in $\{\gam_j'\}$ must have endpoints that lie in the same
	arc out of $\beta_1(x_2,x_3),  \ldots$, $\beta_1(x_{n-1},x_{n})$ and $\beta_1(x_{n},x_2)$.
	Let us call these arcs $\gam_1 = \gam_1[y_1, z_1]$ and $\gam_2 = \gam_2[y_2, z_2]$,
	where $y_1,y_2 \in \beta_1$, and $z_1,z_2 \in \beta_2$.
	
	Let $\gamma_3$ be the simple closed curve formed from $\beta_1[y_1, y_2]$ (containing $x_2, \ldots, x_{n}$),
	$\gam_1, \gam_2$, and $\beta_2[z_1, z_2]$ (containing $x_1$).
	As $\beta_1, \beta_2$ are quasi-arcs, and we have control on the distance of $\gam_1, \gam_2$ from
	$x_1, \ldots, x_n$, we can apply Theorem~\ref{thm-qarc-joining} to straighten $\gam_2$ into
	a $\lambda$-quasi-circle $\gam$, where $\lam=\lam(L,N,\lam_1,\lam_2,\del_*,s/S) = \lam(L,N,n)$.
	Moreover, $\diam(\gam) \leq 4\lam_1 S$ as desired.
	
	\vspace{2mm}
	{\noindent\emph{Case 2:}} Suppose $s < \del^{n-1}$.
	
	This case is similar to Case 1, except now $s$ may be arbitrarily small,
	so we replace $\beta_2$ by a quasi-circle through $x_1$ and all points close to it.
	
	Consider the set $U = \{d(x_1, x_i)\}_{i=3}^{n-1}$ of size $n-3$.
	One of the intersections 
	$U \cap [\del^{n-1}, \del^{n-2}), \ldots, U \cap [\del^2, \del^1)$ is empty.
	Thus there exists 
	$m \in \{2, \ldots, n-1 \}$ so that 
	$d(x_1, x_m) \leq \del d(x_1, x_{m+1})$.
		
	Let $\alp_1$ be a $\lam_1$-quasi-circle through $\{x_1, x_{m+1}, x_{m+2}, \ldots, x_n\}$.
	Similarly to case 1, use the $L$-annularly linearly connected property for
	$A(x_1, 4\lam_1 L d(x_1, x_m), 8\lam_1 L d(x_1, x_m))$ to find a circle $\alp_2$ that detours
	$\alp_1$ around $B(x_1, 4\lam_1 d(x_1, x_m))$, while only cutting out loops
	in \[ B(x_1, 8 \lam_1^2 L^2 d(x_1, x_m)) \subset B(x_1, \tfrac45 d(x_1, x_{m+1})). \]
	In particular, $\alp_2$ contains $\{x_{m+1}, \ldots, x_n\}$, and we relabel so they are in this cyclic order.
	
	We use Theorem~\ref{thm-qarc-joining} with $\eps = \lam_1 d(x_1, x_m)$ to straighten $\alp_2$ into a 
	quasi-circle $\beta_1$ which remains outside $B(x_1, 3\lam_1 d(x_1, x_m))$.
	Moreover, $\beta_1$ will $9 \lam_1 L d(x_1, x_m)$-follow $\alpha_1$.
	Inside $B(x_1, 9 \lam_1^2 L^2 d(x_1, x_m))$, $\beta_1$ is a $\lam_3$-quasi-arc,
	where $\lam_3 = \lam_3(L,N,\lam_1)$ is independent of $d(x_1,x_m)$.
	
	\begin{figure}
		\begin{center}
		\includegraphics[width=.4\textwidth]{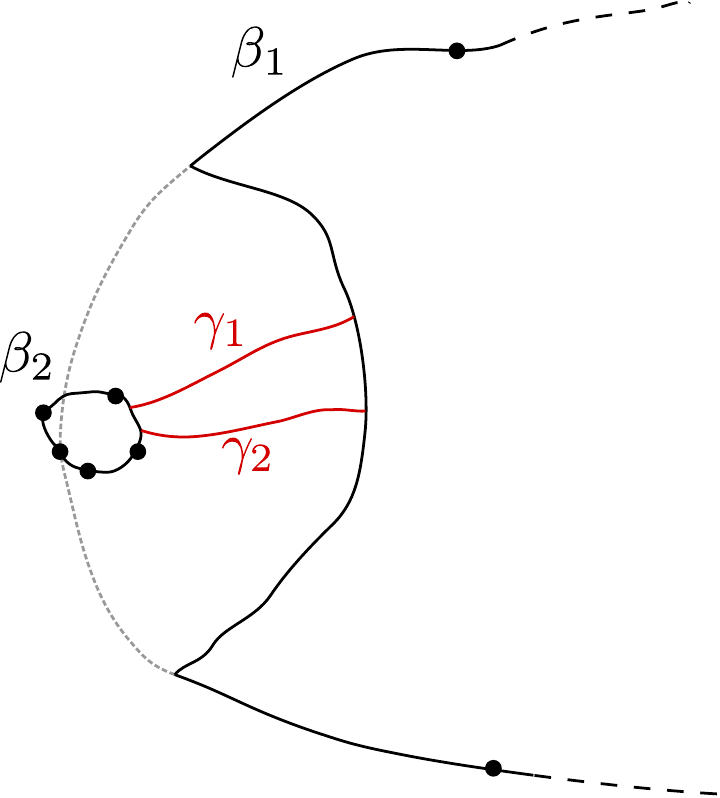}
		\end{center}
 		\caption{Joining two quasi-circles in case 2}
		\label{fig-joincase2}
	\end{figure}
	
	Let $\beta_2$ be a $\lam_1$-quasi-circle through $\{x_1, \ldots, x_m\}$, relabelled so they are in this cyclic order,
	of diameter at most $2\lam_1 d(x_1,x_m)$
	(see Figure~\ref{fig-joincase2}).
	
	As in Case 1, by Theorem~\ref{thm-zippin} we can join 
	$\beta_1$ to $\beta_2$ by $2n$ disjoint arcs inside $B(x_1, 10\lam_1^2 L^2 d(x_1, x_m))$.
	Inside this ball we have control on the diameter of $\beta_2$, and the quasi-arc constants of $\beta_1, \beta_2$.
	Therefore, a similar argument to Lemma~\ref{lem-qarc-join-unif-sep1} gives 
	that these arcs are $\del_* d(x_1, x_m)$ separated, where $\del_* = \del_*(L,N,\lam_1, \lam_3)$.
	
	As before, $n$ of these arcs, let us call them $\gam_1', \ldots, \gam_n'$, will join $\beta_1$ to $\beta_2$,
	be $\delta_* d(x_1, x_m)$-separated, and have distance at least $\frac12 \delta_* d(x_1, x_m)$ from any $x_i$.
	
	By the pigeonhole principle, two of the arcs in $\{\gam_j'\}$ must have endpoints that lie in the same
	arc out of $\beta_2(x_1,x_2), \ldots, \beta_2(x_{m-1},x_{m})$ and $\beta_2(x_{m},x_1)$.
	Let us call these arcs $\gam_1 = \gam_1[y_1, z_1]$ and $\gam_2 = \gam_2[y_2, z_2]$,
	where $y_1,y_2 \in \beta_1$, and $z_1,z_2 \in \beta_2$.  (Again, see Figure~\ref{fig-joincase2}.)
	
	Using the fact that $\beta_1$ follows $\gam_1$, we see that the diameter of the smaller
	arc $\beta_1[y_1, y_2]$ is at most $100\lam_1^3 L^2 d(x_1, x_m) < \frac{1}{2} d(x_1, x_{m+1})$.
	Therefore, there is a subarc $\beta_1'[y_1, y_2] \subset \beta_1$ containing $x_{m+1}, \ldots, x_{n}$.
	
	Let $\gamma_3$ be the simple closed curve formed from $\beta_1'$,
	$\gam_1, \gam_2$, and $\beta_2[z_1, z_2]$ (containing $x_1, \ldots, x_m$).
	As $\beta_1, \beta_2$ are quasi-arcs, and we have control on the distance of $\gam_1, \gam_2$ from
	$x_1, \ldots, x_n$, we can apply Theorem~\ref{thm-qarc-joining} with $\eps = \frac{1}{2} \del_* d(x_1, x_m)$
	to straighten $\gam_2$ into a quasi-circle $\gam$.
	
	Let us show that $\gam$ is a quasi-circle with controlled constant.
	Observe that $\gam$ $D$-follows $\alp_1$, where $D = 10 \lam_1^2 L^2 d(x_1, x_m)$.
	Let $f : \gam \ra \alp_1$ be the associated map. Consider the following three cases.
	
	(i) From Theorem~\ref{thm-qarc-joining}, there exists $\lam_4 = \lam_4(L,N, \lam_1, \del_*)$
	so that if $z,z' \in \gam \cap B(x_1, 10 \lam_1 D)$, then $\diam(\gam[z,z']) \leq \lam_4 d(z,z')$.
	
	(ii) If $\gam[z,z'] \cap B(x_1, 2 \lam_1 D) = \emptyset$, then $\gam[z,z'] = \alp_1[z,z']$, 
	so we have $\diam(\gam[z,z']) \leq \lam_1 d(z,z')$.
	
	(iii) Otherwise, we know that $\diam(\gam[z,z']) \geq 8 \lam_1 D$, so
	\begin{align*}
		\diam(\gam[z,z']) & \leq 2D + \diam(\alp_1[f(z),f(z')])
			\leq 2D + \lam_1 d(f(z),f(z')) \\ &
			\leq 2D + 2\lam_1 D + \lam_1 d(z,z'),
	\end{align*}
	so $\frac{1}{2} \diam(\gam[z,z']) \leq \lam_1 d(z,z')$, thus
	$\diam(\gam[z,z']) \leq 2 \lam_1 d(z,z')$.
		
	Therefore, $\gam$ is a $\lam$-quasi-circle, where $\lam = \max\{\lam_4, 2\lam_1\}$ depends only on $L,N,n$.
	Observe that $\diam(\gam) \leq 4\lam_1 S$, as desired.	
\end{proof}

\bibliographystyle{alpha}
\bibliography{biblio}

\end{document}